\documentclass{article}
\usepackage[multiple]{footmisc}
\usepackage[utf8]{inputenc}
\usepackage[T1]{fontenc}
\usepackage[left=1.8cm,right=1.8cm,top=1.8cm,bottom=1.8cm]{geometry}
\usepackage{tikz}
\usepackage{enumerate,mathtools}
\usepackage{amsmath}
\usepackage{amssymb}
\usepackage{amsthm,dsfont}
\usepackage{enumitem} 
\newtheorem{thm}{Theorem}[section]
\newtheorem{lem}[thm]{Lemma}
\theoremstyle{remark}

\numberwithin{equation}{section}

\newcommand\ens[1]{\{1,\dots,#1\}}
\title{On the rate of convergence of the shape of Young diagrams associated with random words}
\author{Clément Deslandes\footnote{C.M.A.P. \'Ecole Polytechnique, Palaiseau, 91120, France \& Georgia Institute of Technology, Atlanta, GA, 30332, USA (\texttt{clement.deslandes@poytechnique.edu}).}\qquad Christian Houdré\footnote{School of Mathematics, Georgia Institute of Technology, Atlanta, GA, 30332, USA (\texttt{houdre@math.gatech.edu}).} \footnote{Research supported in part by the grant $\sharp 524678$ from the Simons Foundation. \newline\indent
		}}

\begin{document}

\maketitle
\begin{abstract}
We revisit, beyond the uniform case, some aspects of the convergence of the cumulative shape of the RSK Young diagrams associated with random words, obtaining rates of convergence in Kolmogorov's distance.  Since the length of the top row of the diagrams is the length of the longest increasing subsequences of the word, a corresponding rate result follows.  This is then extended to the length of the longest common and increasing subsequences  in two or more random words.   
\end{abstract}
\section{Introduction}

Let $X_1,\dots,X_n$ be i.i.d.~random variables with values in a finite alphabet $\ens{m}$ and with probability mass function given by $p_1,\dots,p_m$. 
For $i\in\ens{m}$ and $j\in \{0,\dots,n\}$, and with elements of the notation in \cite{HL, BH, DH1}, let $N^{i}_{j}=\sum_{l=1}^j \mathds{1}_{X_l=i}$ be the number of letters $i$ within the 
first $j$ letters, and for $t\in [0,1]$, let \begin{equation}\label{defB}
\widetilde{B^n_i}(t)=\frac{N^{i}_{\lfloor tn \rfloor}-p_i \lfloor tn \rfloor}{\sigma_i\sqrt{n}},
\end{equation}
where $\sigma_i=\sqrt{p_i(1-p_i)}$.
Let $\Lambda=\{\lambda\in [0,1]^m : \lambda_1+\dots+\lambda_m=1\}$, and 
for $\lambda\in\Lambda$, let

\begin{equation}
\widetilde{V^{n}_i}(\lambda)=\sigma_i\left(\widetilde{B^n_i}(\lambda_1+\dots+\lambda_i)-\widetilde{B^n_i}(\lambda_1+\dots+\lambda_{i-1})\right),
\end{equation}
$i\in \{1, \dots, m\}$ with for $i=1$, the convention that $\widetilde{B^n_i}(\lambda_1+\dots+\lambda_{i-1}) 
= 0$. 

\noindent
Hence,  $LI_n$, the length of the longest increasing subsequences of $X_1\cdots X_n$ is given by:

\begin{equation}\label{exprlcin}
LI_n=\max_{0=k_0\leq k_1\leq \dots\leq k_m=n} \sum_{i=1}^m \left(N^i_{k_i}-N^i_{k_{i-1}}\right)=\max_{\substack {\lambda\in \Lambda_d}} \sum_{i=1}^m 
\left(n p_i\lambda_i+\sqrt{n}\widetilde{V^{n}_i}(\lambda)\right),
\end{equation}
where now $\Lambda_d:=\{\left({j_1}/{n},\dots,{j_m}/{n}\right):j_1,\dots,j_m\in \mathbb{N}, j_1+\dots+j_m=n\}$.

\noindent
For $\lambda\in\Lambda$ let, finally,
\begin{equation}\label{defZ}
Z_n(\lambda):=\sum_{i=1}^m \left(\sqrt{n} (p_i-p_{\max})\lambda_i+\widetilde{V^{n}_i}(\lambda)\right),
\end{equation}

\noindent
where $p_{\max} =  \max_{i=1, \dots, m}p_i$, so that 
\begin{equation}\label{exprlcinrec}
\frac{LI_n-np_{\max}}{\sqrt{n}}=\max_{\substack {\lambda\in \Lambda_d}}Z_n(\lambda).
\end{equation}

It is known (see \cite{HL,BH, DH1} and the references therein) that the limiting distribution of \eqref{exprlcinrec} is the distribution of $\max_{\substack {\lambda\in \Lambda}}Z'_n(\lambda)$ where $Z'_n$ is defined as  $Z_n$ but with $B^n_i$, a Brownian motion, instead of $\widetilde{B^n_i}$ (as stated more precisely in the sequel). Note that $Z'_n$ has the same distribution for all $n$, so that the limiting distribution above is well defined. Below, our main goal is to provide a rate of convergence for this result.  To do so, the  strategy is to use a KMT approximation to build a coupling: we will define on the same probability space $\widetilde{B}$ and $B$ that are very close. Then, when the letters are uniformly distributed, $Z_n(\lambda)$ simplifies to $Z_n(\lambda)= \sum_{i=1}^m \widetilde{V^{n}_i}(\lambda)$, and therefore it is straightforward to infer from the coupling the rate of convergence bound.  This can also be done for the other 
lines of the RSK Young diagrams associated with the word.  When the distribution is not uniform, the  strategy is to first approximate $\max_{\substack {\lambda\in \Lambda_d}}Z_n(\lambda)$ by $\max_{\substack {\lambda\in \Lambda^\prime_d}}Z_n(\lambda)$, where $\Lambda^\prime_d=\{\lambda\in \Lambda_d : p_i\neq p_{\max}\implies \lambda_i=0\}$, because then for each $\lambda\in \Lambda^\prime_d$, $Z_n(\lambda)=\sum_{i=1}^m \widetilde{V^{n}_i}(\lambda)$, as previously.  These approaches also give the convergence of moments.  To finish, the case of two or more common and increasing subsequences is addressed.  

\section{A coupling via KMT and rates of convergence}

To start with, we prove the following key coupling lemma:  
\begin{lem}\label{lemKMTsup}
Let $\alpha\geq 1$. For every $m\geq 2$, every probability mass function $p_1,\dots,p_m$, and every $n\ge 2$,  there exists a probability space with  $X_1,\dots,X_n, \widetilde{B^n}$, as above, and $B^n$ an $m$-dimensional Brownian motion with covariance matrix  $\Sigma:=\mathrm{Cov}\left((\mathds{1}_{X_1=i}/\sigma_i)_{1\leq i\leq m}\right)$, defined on it, such that
\begin{equation}\label{KMTsup}
\mathds{P}\left(\sup_{\substack{i \in\ens{m} \\ 0\leq t \leq 1}} \left| \sigma_i \widetilde{B^n_i}(t)-\sigma_i B^n_i(t)\right| \geq C\frac{\alpha(\log n)^2}{\sqrt{n}} \right)\leq \frac{2}{n^\alpha},
\end{equation}
where $C$ is a universal constant.
\end{lem}

\begin{proof}
By the KMT approximation \cite{komlos1975approximation} (see also \cite{shorack2009empirical} for further details and extensive references, in particular on Kiefer processes), there exist a probability space with $(U_i)_{i\geq1}$ i.i.d. uniform on $[0,1]$ random variables and a Kiefer process $(K(s,t))_{\substack{s\in[0,1] \\ t\in[0,\infty)}}$ such that for all $x>0$, 
\begin{equation}
	\mathds{P}\left(\sup_{\substack{0\leq s\leq 1 \\ l\in\ens{n}}} \left|\frac{1}{\sqrt{n}} \sum_{k=1}^l \left(\mathds{1}_{U_k \leq s}-s\right) -\frac{K\left(s,l \right)}{\sqrt{n}} \right| \geq C_1\log n \frac{\log n + x}{\sqrt{n}} \right)\leq \mathrm{e}^{-x},
\end{equation}
where $C_1$ is a universal constant (throughout, $C_2,C_3, \dots$ are universal constants). In particular, for $x=\alpha \log n$ one gets

\begin{equation}\label{disc}
	\mathds{P}\left(\sup_{\substack{0\leq s\leq 1 \\ l\in\ens{n}}} \left|\frac{1}{\sqrt{n}} \sum_{k=1}^l \left(\mathds{1}_{U_k \leq s}-s\right) -\frac{K\left(s,l \right)}{\sqrt{n}} \right| \geq 2 C_1 \alpha \frac{(\log n )^2}{\sqrt{n}} \right)\leq \frac{1}{n^\alpha}.
\end{equation}
\noindent
To replace the discrete parameter $l\in\ens{n}$ by a continuous one $l'\in [0,n]$, note that 

\begin{align*}
\mathds{P}\left(\sup_{\substack{0\leq s\leq 1 \\ l\in\ens{n}}} \sup_{l'\in[l-1,l]}\left|K\left(s,l \right)-K\left(s,l' \right) \right| \geq 2C_1 \alpha (\log n )^2 \right)& \leq \sum_{l=1}^n \mathds{P}\left(\sup_{\substack{0\leq s\leq 1 \\ l'\in[l-1,l] }}\left|K\left(s,l \right)-K\left(s,l' \right) \right| \geq 2C_1 \alpha (\log n )^2 \right)\\
&\leq n \mathds{P}\left(\sup_{\substack{0\leq s\leq 1 \\ 0\leq t \leq 1 }}\left|K\left(s,t \right) \right| \geq 2C_1 \alpha (\log n )^2 \right)\\
&\leq n \mathds{P}\left(2\sup_{\substack{0\leq s\leq 1 \\ 0\leq t \leq 1 }}\left|W\left(s,t \right) \right| \geq 2C_1 \alpha (\log n )^2 \right),
\end{align*}
where $W$ is a two-dimensional Brownian sheet (using the facts that $(s,t)\mapsto K(s,l)-K(s,l-t)$ and $(s,t)\mapsto W(s,t)-sW(1,t)$ are Kiefer processes on $[0,1]^2$). From \cite[Theorem 3]{goodman1976distribution}, and if $\Phi$ is the standard normal cumulative distribution function, 

\begin{equation}\label{thm3}
\mathds{P}\left(\sup_{\substack{0\leq s\leq 1 \\ 0\leq t \leq 1 }}\left|W\left(s,t \right) \right| \geq C_1 \alpha (\log n )^2 \right)\leq 4 \Phi(-C_1 \alpha (\log n )^2).
\end{equation} 
If $C_1$ is large enough, which can be assumed without loss of generality, for all $\alpha\geq 1$ and all $n\geq 2$, $4 \Phi(-C_1 \alpha (\log n )^2)\leq {1}/{n^{1+\alpha}}$. Therefore,
\begin{equation}\nonumber
\mathds{P}\left(\sup_{\substack{0\leq s\leq 1 \\ l\in\ens{n}}} \sup_{l'\in[l-1,l]}\left|K\left(s,l \right)-K\left(s,l' \right) \right| \geq 2C_1 \alpha (\log n )^2 \right)\leq \frac{n}{n^{1+\alpha}}=\frac{1}{n^\alpha}, 
\end{equation}
and using \eqref{disc},
\begin{equation}\label{KMTsupbridge}
	\mathds{P}\left(\sup_{\substack{0\leq s\leq 1 \\ 0\leq t \leq 1}} \left|\frac{1}{\sqrt{n}} \sum_{k=1}^{\lfloor tn \rfloor} \left(\mathds{1}_{U_k \leq s}-s\right) -\frac{K\left(s,tn \right)}{\sqrt{n}} \right| \geq 4C_1\alpha\frac{(\log n)^2}{\sqrt{n}} \right)\leq \frac{2}{n^\alpha}.
\end{equation}

For $i\in\ens{n}$, let $X_i:=\min_{k\in\ens{m}} \{k: U_i\leq p_1+\dots+p_k\}$.  Clearly, the $X_i$ are  i.i.d random variables with values in $\ens{m}$ and probability mass function $p_1,\dots,p_m$. So,  with the notations above,
\begin{equation}\nonumber
\widetilde{B^n_i}(t)=\frac{N^{i}_{1,\lfloor tn \rfloor}-p_i \lfloor tn \rfloor}{\sigma_i\sqrt{n}}=\frac{\sum_{k=1}^{\lfloor tn \rfloor} \left(\mathds{1}_{U_k \leq p_1+\dots+p_i}-(p_1+\dots+p_i)-\mathds{1}_{U_k \leq p_1+\dots+p_{i-1}}+p_1+\dots+p_{i-1}\right)}{\sigma_i\sqrt{n}}.
\end{equation}
For $i\in\ens{m}$ and $t\in[0,1]$, then 
\begin{equation*}
B^n_i(t):=\frac{K\left(p_1+\dots+p_i,tn \right)-K\left(p_1+\dots+p_{i-1},tn \right)}{\sigma_i\sqrt{n}},
\end{equation*}
are Brownian motions with covariance matrix $\Sigma:=\mathrm{Cov}\left((\mathds{1}_{X_1=i}/\sigma_i)_{1\leq i\leq m}\right)$.  Note that \begin{equation}\nonumber
\sup_{\substack{i\in \ens{m} \\ 0\leq t\leq 1}} \left| \sigma_i \widetilde{B^n_i}(t)-\sigma_i B^n_i(t)\right|\leq 2\sup_{\substack{0\leq s\leq 1 \\ 0\leq t \leq 1}} \left|\frac{1}{\sqrt{n}} \sum_{k=1}^{\lfloor tn \rfloor} \left(\mathds{1}_{U_k \leq s}-s\right) -\frac{K\left(s,tn \right)}{\sqrt{n}} \right|, 
\end{equation}
and so from \eqref{KMTsupbridge} the following coupling inequality: 
\begin{equation}\nonumber
\mathds{P}\left(\sup_{\substack{i\in \ens{m} \\ 0\leq t\leq 1}} \left| \sigma_i \widetilde{B^n_i}(t)-\sigma_i B^n_i(t)\right| \geq 8 C_1 \frac{\alpha(\log n)^2}{\sqrt{n}} \right)\leq \frac{2}{n^{\alpha}},
\end{equation}
is valid and this gives the desired result (letting $C=8 C_1$).

\end{proof}

{\it From now on, our setting is the probability space 
	introduced in Lemma~\ref{lemKMTsup} with its notation.}

To start with, we address the case of uniformly distributed letters and study the rate of convergence, in Kolmogorov distance, for the cumulative shape of the RSK Young diagrams associated with the random word.   



Let $(R_k(n,m))_{1\le k \le n}$ be the shape of the RSK Young diagram associated with the random words $X_1\cdots X_n$ with uniformly distributed letters over $\{1, \dots, m\}$, and define for $1\leq k\leq m$, $V_k(n,m)=\sum_{l=1}^k R_l(n,m)$ (so that, for example, $V_1(n,m)=R_1(n,m)=LI_n$). 
From \cite[Corollary 3.1]{HL}, for properly defined $I_{k,m}$ (keeping the notations of \cite{HL}), $(V_k(n,m))_{1\le k \le n}$ is such that:  
\begin{equation}\label{firstlimitlaw}
\left(\frac{V_{k}(n,m)-kn/m}{\sqrt{n/m}}\right)_{1\leq k\leq m}\xRightarrow[n\to \infty]{}\left(\max_{\mathbf{t}\in I_{k,m}}\sum_{j=1}^k\sum_{l=j}^{m-k+j}\sqrt{(m-1)/m}\left(B_l(t_{j,l})-B_l(t_{j,l-1})\right)\right)_{1\leq k\leq m},
\end{equation}
where the convergence is in distribution, and $B$ is a $m$-dimensional Brownian motion with covariance matrix having diagonal terms equal to $1$ and off-diagonal terms equal to $-1/(m-1)$, i.e., the covariance matrix $\Sigma$ of Lemma~\ref{lemKMTsup}.  Recall finally that the limiting law in \eqref{firstlimitlaw} is the spectra of a traceless $m\times m$ GUE matrix.

To simplify notations, let further denote $\left(\max_{\mathbf{t}\in I_{k,m}}\sum_{j=1}^k\sum_{l=j}^{m-k+j}\sqrt{(m-1)/m}\left(B^n_l(t_{j,l})-B^n_l(t_{j,l-1})\right)\right)_{1\leq k\leq m}$ by $(J_{1,m},\dots,J_{m,m})$ and $\left(({V_{k}(n,m)-kn/m})/{\sqrt{n/m}}\right)_{1\leq k\leq m}$ by $(T_{1,m},\dots,T_{m,m})$.

\begin{thm} 
For every $n,m\geq 2$ and $1\leq k \leq m$, 
$$\sup_{x\in\mathbb{R}}\left|\mathds{P}\left(T_{k,m}\geq x\right)-\mathds{P}\left(J_{k,m}\geq x\right)\right|\leq C(m)\frac{(\log n)^2}{\sqrt{n}},$$
where $C(m)$ is a constant only depending on $m$.
\end{thm}

\begin{proof}

As shown in  \cite[(3.7)]{HL}, 
\begin{equation}
T_{k,m}=\max_{\mathbf{t}\in I_{k,m}}\sum_{j=1}^k\sum_{l=j}^{m-k+j}\sqrt{(m-1)/m}\left(\widetilde{B^n_l}(t_{j,l})-\widetilde{B^n_l}(t_{j,l-1})\right), 
\end{equation}
so applying \eqref{KMTsup}, since $\left|T_{k,m}-J_{k,m}\right|\leq 2k(m-k+1)\sqrt{m} \sup_{\substack{i \in\ens{m} \\ 0\leq t \leq 1}} \left| \sigma_i \widetilde{B^n_i}(t)-\sigma_i B^n_i(t)\right|$, 
\begin{equation}\label{boundtkm}
\mathds{P}\left(\left|T_{k,m}-J_{k,m}\right|\geq 2k(m-k+1)\sqrt{m}C\alpha\frac{(\log n)^2}{\sqrt{n}}\right)\leq \frac{2}{n^{\alpha}}.
\end{equation}

Recall next that with $\Theta_k:\mathbb{R}^k \to \mathbb{R}^k $ given by $(\Theta_k(x))_j = \sum_{i=1}^jx_i$, $1 \le j \le k$, it follows that  $\Theta_m^{-1}(J_{1,m},\dots,J_{m,m})$ has the same distribution as the ordered spectrum of an $m\times m$ traceless GUE matrix. Moreover, by a bound on the joint density of the eigenvalues (see (3.5) in \cite{HT}), there exists $D(m)>0$ bounding the supremum of the density of $J_{k,m}$ (for $1\leq k\leq m$). 

\noindent 
So, for any $x\in\mathbb{R}$, 
\begin{align*}
\left|\mathds{P}\left(T_{k,m}\geq x\right)-\mathds{P}\left(J_{k,m}\geq x\right)\right| & \leq \mathds{P}\left(\left|T_{k,m}-J_{k,m}\right|\geq 2k(m-k+1)\sqrt{m}C\alpha\frac{(\log n)^2}{\sqrt{n}}\right)\\ &\qquad \qquad\qquad + \mathds{P}\left(\left|J_{k,m}-x\right| \leq 2k(m-k+1)\sqrt{m}C\alpha\frac{(\log n)^2}{\sqrt{n}}\right)\\
 & \leq \frac{2}{n^{\alpha}}+D(m)2k(m-k+1)\sqrt{m}C\alpha\frac{(\log n)^2}{\sqrt{n}},
\end{align*}
\noindent
so $\sup_{x\in\mathbb{R}}\left|\mathds{P}\left(T_{k,m}\geq x\right)-\mathds{P}\left(J_{k,m}\geq x\right)\right|$ is upper bounded as stated.

\end{proof} 

As a corollary, we can also study the speed of convergence of $T_{.,m}=(T_{k,m})_{1\leq k\leq m}$ towards $J_{.,m}=(J_{k,m})_{1\leq k\leq m}$, in the Kolmogorov distance, rather than coordinate by coordinate. Just as before, we have 

\begin{equation}
\mathds{P}\left(\left\|T_{.,m}-J_{.,m}\right\|_{\infty}:=\max_{1\leq k \leq m}|T_{.,m}-J_{.,m}|\geq 2m^{5/2}C\alpha\frac{(\log n)^2}{\sqrt{n}}\right)\leq \frac{2}{n^{\alpha}},
\end{equation}

\noindent
and

\begin{equation}
\mathds{P}\left(\max_{1\leq k \leq m}\left|J_{k,m}-x\right| \leq 2m^{5/2}C\alpha\frac{(\log n)^2}{\sqrt{n}}\right)\leq mD(m)2m^{5/2}C\alpha\frac{(\log n)^2}{\sqrt{n}},
\end{equation}

\noindent
so for any $x_1,\dots,x_m\in\mathbb{R}$, 

\begin{eqnarray*}
\left|\mathds{P}\left(\max_{1\leq k\leq m} T_{k,m}\geq x_k\right)-\mathds{P}\left(\max_{1\leq k \leq m} J_{k,m}\geq x_k\right)\right|  \leq & \!\!\! \!\!\! \!\!\! \!\!\! \!\!\! \!\!\! \!\!\! \!\!\!\mathds{P}\left(\left\|T_{.,m}-J_{.,m}\right\|_{\infty}\geq 2m^{5/2}C\alpha\frac{(\log n)^2}{\sqrt{n}}\right)\\ 
&  + \mathds{P}\left(\max_{1\leq k\leq m}\left|J_{k,m}-x\right| \leq 2m^{5/2}C\alpha\frac{(\log n)^2}{\sqrt{n}}\right)\\
& \!\!\! \!\!\! \!\!\! \!\!\! \!\!\! \!\!\! \!\!\! \!\!\! \!\!\! \!\!\! \!\!\! \!\!\! \le \frac{2}{n^{\alpha}} + mD(m)2m^{5/2}C\alpha\frac{(\log n)^2}{\sqrt{n}}.
\end{eqnarray*}

As in \cite{BH2}, for $m$ no longer fixed, let us consider conditions on the sequences  ($m(n))_{n\geq 1}$ (writing just $m$ for $m(n)$) under which for all $\varepsilon>0$,   

\begin{equation}\label{limitgoal}
\mathds{P}\left(\left| \left(T_{1,m}-2\sqrt{m}\right)m^{1/6}-\left(J_{1,m}-2\sqrt{m}\right)m^{1/6}\right|\geq \varepsilon\right) \xrightarrow[n\to \infty]{}  0.
\end{equation}

Then, since $\left(J_{1,m}-2\sqrt{m}\right)m^{1/6}$ converges in distribution to the Tracy--Widom distribution $F_{TW}$, this implies that $\left(T_{1,m}-2\sqrt{m}\right)m^{1/6}$, that is, $\left(\left(LI_n-n/m\right)/\sqrt{n/m}-2\sqrt{m}\right)m^{1/6}$, converges to $F_{TW}$ as well. 

Applying \eqref{boundtkm} gives 
\begin{equation}
\mathds{P}\left(\left| \left(T_{1,m}-2\sqrt{m}\right)m^{1/6}-\left(J_{1,m}-2\sqrt{m}\right)m^{1/6}\right|\geq Cm^{1/6}2k(m-k+1)\sqrt{m}\frac{\alpha(\log n)^2}{\sqrt{n}}\right)\leq \frac{2}{n^{\alpha}}, 
\end{equation}
so that if $m^{1/6}2k(m-k+1)\sqrt{m}C\alpha{(\log n)^2}/{\sqrt{n}}$ converges to zero, that is, if $m=o((n/(\log n)^4)^{3/10})$, then  \eqref{limitgoal} follows, and therefore the convergence in distribution of the properly centered and scaled $LI_n$ to the Tracy-Widom distribution also follows.


Beyond the uniform case, in order to evaluate the rate of convergence to the limiting law for arbitrary distributions, we first need to control how close $\max_{\substack {\lambda\in \Lambda_d}}Z_n(\lambda)$ is to $\max_{\substack {\lambda\in \Lambda^\prime_d}}Z_n(\lambda)$ where again,$$\Lambda^\prime_d:=\{\left({j_1}/{n},\dots,{j_m}/{n}\right):j_1+\dots+j_m=n\text{ and }p_i\neq p_{\max}\implies j_i=0\}.$$

\begin{lem}\label{lemnonuni}
Let $n,m\geq 2, \alpha\geq 1, a=6+3\alpha$ and $\Delta=p_{\max}-p_{2nd}$, where $p_{2nd}$ is the second highest of the $p_i$'s. Let $a\log n\leq 2\sqrt{n}\Delta$. Then,
$$\mathds{P}\left(\left| \frac{\max_{\substack {\lambda\in \Lambda_d}}Z_n(\lambda)}{\sqrt{p_{\max}}}-\frac{\max_{\substack {\lambda\in \Lambda'_d}}Z_n(\lambda)}{\sqrt{p_{\max}}}\right|> \frac{(a\log n)^2}{4\Delta \sqrt{n p_{\max}}}+\frac{am\log n}{\sqrt{n p_{\max}}}\right)\leq \frac{2m}{n^\alpha}.$$
\end{lem}

\begin{proof}

Our analysis of this result rests upon estimating the variations of $\widetilde{B^n_i}$. To do so, let $A_n$ be the event: 
\begin{equation*}
\forall i\in\ens{m}, \forall j \in \{0,\dots,n-1\}, \forall \ell\in\{1,\dots,n-j\}, \left\vert \frac{N^{i}_{j+\ell}-N^i_j-p^X_i \ell}{\sqrt{n}}\right\vert\leq \frac{\sigma_i \sqrt{\ell}+1}{\sqrt{n}}a\log{n}.
\end{equation*}
By Bernstein's  inequality:
\begin{equation}\label{azuma}
1-\mathds{P}\left(A_n\right)\leq 2m\frac{n(n+1)}{2} \exp\left(-\frac{\frac{1}{2}(\sigma_i \sqrt{\ell}+1)^2 (a\log n)^2}{\ell \sigma^2_i+\frac{1}{3}(\sigma_i \sqrt{\ell}+1)a\log{n}}\right)
\end{equation}
and since
\begin{equation}
\frac{\frac{1}{2}(\sigma_i \sqrt{\ell}+1)^2 (a\log n)^2}{\ell \sigma^2_i+\frac{1}{3}(\sigma_i \sqrt{\ell}+1)a\log{n}}\geq 
\frac{\frac{1}{2}(\sigma_i \sqrt{\ell}+1)^2 (a\log n)^2}{\frac{2}{3}(\sigma_i \sqrt{\ell}+1)a\log{n}+\frac{1}{3}(\sigma_i \sqrt{\ell}+1)a\log{n}}\wedge
\frac{\frac{1}{2}(\sigma_i \sqrt{\ell}+1)^2 (a\log n)^2}{\ell \sigma^2_i+\frac{1}{2}\ell \sigma^2_i}\geq \frac{a}{3}\log{n},
\end{equation}
it follows that
\begin{equation*}
1-\mathds{P}\left(A_n\right) \leq 2mn^{2-\frac{a}{3}}\leq \frac{2m}{n^{\alpha}}.
\end{equation*}
If $A_n$ occurs, then for all $\lambda\in \Lambda_d$ and $i\in \ens{m}$, 
\begin{equation}\label{boundV}
\widetilde{V^{n}_i}(\lambda)\leq a\log{n}\left(\sigma_i \sqrt{\lambda_i}+\frac{1}{\sqrt{n}}\right).
\end{equation}
Let $A_n$ occur. Let $\lambda\in \Lambda_d$. Let $s=\sum_{j: p_j\neq p_{\max}} \lambda_j$, let $i\in\ens{m}$ be such that $p_i=p_{\max}$ and let $\lambda'$ be defined as: $\lambda'_j=0$ for all $j$ such that $p_j\neq p_{\max}$, $\lambda'_i=\lambda_i+s$, and $\lambda'_j=\lambda_j$ elsewhere. So $\lambda^\prime\in\Lambda^\prime_d$ and from \eqref{boundV},
\begin{equation*}
Z_n(\lambda')\geq Z_n(\lambda)+\sqrt{n}s\Delta - \left(\sigma_i\sqrt{s}+\frac{1}{\sqrt{n}}\right)a\log n - \sum_{j:p_j\neq p_{\max}}\left(\sigma_j\sqrt{\lambda_j}+\frac{1}{\sqrt{n}}\right)a\log n,
\end{equation*}
which leads by the Cauchy-Schwarz inequality to
\begin{equation*}
Z_n(\lambda')\geq Z_n(\lambda)+\sqrt{n}s\Delta- \sqrt{s}a\log n-\frac{am\log n}{\sqrt{n}}.
\end{equation*}
Hence, since $a\log n\leq 2\sqrt{n}\Delta$,
\begin{equation*}
Z_n(\lambda')\geq Z_n(\lambda)-\frac{(a\log n)^2}{4\Delta \sqrt{n}}-\frac{am\log n}{\sqrt{n}},
\end{equation*}
and finally
\begin{equation}
0\leq \max_{\substack {\lambda\in \Lambda_d}}Z_n(\lambda)-\max_{\substack {\lambda\in \Lambda'_d}}Z_n(\lambda) \leq \frac{(a\log n)^2}{4\Delta \sqrt{n}}+\frac{am\log n}{\sqrt{n}}.
\end{equation}
Therefore,
\begin{equation}
\mathds{P}\left(\left| \frac{\max_{\substack {\lambda\in \Lambda_d}}Z_n(\lambda)}{\sqrt{p_{\max}}}-\frac{\max_{\substack {\lambda\in \Lambda'_d}}Z_n(\lambda)}{\sqrt{p_{\max}}}\right|> \frac{(a\log n)^2}{4\Delta \sqrt{n p_{\max}}}+\frac{am\log n}{\sqrt{n p_{\max}}}\right)\leq 1-\mathds{P}\left(A_n\right)\leq \frac{2m}{n^{\alpha}}.
\end{equation}
\end{proof}

We can now deduce our theorem, where below $J_k$ is defined as in \cite[Theorem 4.1]{HT} ($J_k$ has same law as $J_{1,k}+\sqrt{(1-kp_{\max})/k}Z$, where $Z\sim \mathcal{N}(0,1)$ is independent from the other variables)
\begin{thm}\label{thm:newthm41}
Let $n,m\geq 2$ and let $k$ be the multiplicity of $p_{\max}$, then,
\begin{equation}\label{newthm41}
\left|\mathds{P}\left(\frac{LI_n-np_{\max}}{\sqrt{n p_{\max}}}\geq x\right)-\mathds{P}\left(J_k\geq x\right)\right| \leq 2 D(k,p_{\max})\frac{(\log n)^2}{\sqrt{n p_{\max}}}\left(\frac{21}{\Delta}+C_2 m\right),
\end{equation}
where $C_2$ is a universal constant, where  $D(1,p_{\max}):=1/\sqrt{2\pi(1-p_{\max})}$, and where for $k\ge 2$,  $D(k,p_{\max}):=\min\left\{\sqrt{{k}/{2\pi(1-kp_{\max})}}, k^{3k}(2\pi e^2)^{k/2}\sqrt{{e}/{\pi}} \right\}$.
\end{thm}

\begin{proof}

We apply Lemma~\ref{lemKMTsup}, and for $\lambda\in \Lambda^\prime:=\{\lambda\in \Lambda: p_i\neq p_{\max}\implies \lambda_i=0\}$, let $Z'_n(\lambda)$ be defined as $Z_n(\lambda)$ but with $B^n_i$ instead of $\widetilde{B^n_i}$. For any $\lambda\in \Lambda'$, $\left|Z_n(\lambda)-Z'_n(\lambda)\right|\leq 2 m\sup_{\substack{i \in\ens{m} \\ 0\leq t \leq 1}} \left| \sigma_i \widetilde{B^n_i}(t)-\sigma_i B^n_i(t)\right|$, hence 
\begin{equation}\label{bound2}
\mathds{P}\left(\left| \sup_{\substack {\lambda\in \Lambda'}}Z_n(\lambda)-\sup_{\substack {\lambda\in \Lambda'}}Z'_n(\lambda)\right|\geq \frac{2C\alpha m(\log n)^2}{\sqrt{n}}\right)\leq \frac{2}{n^\alpha}.
\end{equation}
For all $\lambda\in\Lambda'$, $Z_n(\lambda)=\sum_{i=1}^m \widetilde{V^n_i}(\lambda)$, so $\sup_{\substack {\lambda\in \Lambda'}}Z_n(\lambda)=\max_{\substack {\lambda\in \Lambda'_d}}Z_n(\lambda)$. So

\begin{equation}
\mathds{P}\left(\left| \max_{\substack {\lambda\in \Lambda'_d}}Z_n(\lambda)-\max_{\substack {\lambda\in \Lambda'}}Z'_n(\lambda)\right|\geq \frac{2C\alpha m(\log n)^2}{\sqrt{n}}\right)\leq \frac{2}{n^\alpha}.
\end{equation}

\noindent
For any $x\in\mathbb{R}$, $\alpha\geq 1$, and $a=6+3\alpha$,

\begin{align*}
\resizebox{0.4\hsize}{!}{$\left|\mathds{P}\left(\frac{\max_{\substack {\lambda\in \Lambda_d}}Z_n(\lambda)}{\sqrt{p_{\max}}}\geq x\right)-\mathds{P}\left(\frac{\max_{\substack {\lambda\in \Lambda'}}Z'_n(\lambda)}{\sqrt{p_{\max}}}\geq x\right)\right|$} &\leq \mathds{P}\left(\left| \frac{\max_{\substack {\lambda_d \in \Lambda}}Z_n(\lambda)}{\sqrt{p_{\max}}}-\frac{\max_{\substack {\lambda\in \Lambda'_d }}Z_n(\lambda)}{\sqrt{p_{\max}}}\right|> \frac{(a\log n)^2}{4\Delta \sqrt{n p_{\max}}}+\frac{am\log n}{\sqrt{n p_{\max}}}\right)\\
&+\mathds{P}\left(\left| \frac{\max_{\substack {\lambda\in \Lambda'_d}}Z_n(\lambda)}{\sqrt{p_{\max}}}-\frac{\max_{\substack {\lambda\in \Lambda'}}Z'_n(\lambda)}{\sqrt{p_{\max}}}\right| >\frac{2C\alpha m(\log n)^2}{\sqrt{n p_{\max}}} \right)\\
& + \mathds{P}\left(\left|\frac{\max_{\substack {\lambda\in \Lambda'}}Z'_n(\lambda)}{\sqrt{p_{\max}}}-x\right|\leq \frac{(a\log n)^2}{4\Delta \sqrt{n p_{\max}}}+\frac{am\log n}{\sqrt{n p_{\max}}}+\frac{2C\alpha m(\log n)^2}{\sqrt{n p_{\max}}} \right).
\end{align*}

Now, with the notations of \cite[Theorem 4.1]{HT}, we see that $\max_{\substack {\lambda\in \Lambda'}}Z^\prime_n(\lambda)/\sqrt{p_{\max}}$ has the same law as $J_k$. Indeed, one can check that in \cite{HT}, $\tilde{B}$ is an $(m-1)$-dimensional Brownian motion with the same covariance  as \linebreak $\left(\!\left(\sigma_r B^n_r - \sigma_{r+1} B^n_{r+1}\right)\!/\!\sqrt{p_r+p_{r+1}-(p_r-p_{r+1})^2}\right)_{1\leq r\leq m}$, and rewriting $\tilde{B}$ that way gives  exactly ${\max_{\substack {\lambda\in \Lambda'}}Z'_n(\lambda)}/{\sqrt{p_{\max}}}$. Assuming $a\log n\leq 2\sqrt{n}\Delta$, from \ref{lemnonuni}, \cite[Proposition 3.1 (ii)]{HT} and the inequality $a=6+3\alpha\leq 9\alpha$, one gets

\begin{align}\label{newthm41alpha}
\left|\mathds{P}\left(\frac{LI_n-np_{\max}}{\sqrt{n p_{\max}}}\geq x\right)-\mathds{P}\left(J_k\geq x\right)\right| &\leq \frac{4m}{n^{\alpha}}+2D(k,p_{\max})\frac{(\log n)^2}{\sqrt{n p_{\max}}}\left(\frac{21\alpha^2}{\Delta}+\left(2C+\frac{9}{\log 2}\right)\alpha m\right),
\end{align}
where we refer to \cite[Proposition 3.1]{HT} for a proof that $D(k,p_{\max})$ is a bound on the infinity norm of the density of the limiting distribution $\max_{\substack {\lambda\in \Lambda'}}Z'_n(\lambda)$. Taking $\alpha=1$ and using  $D(k,p_{\max})\geq 1/\sqrt{2\pi}$, one gets, in particular, \eqref{newthm41}. Now if we no longer assume that $a\log n\leq 2\sqrt{n}\Delta$, the bound \eqref{newthm41alpha} (and therefore \eqref{newthm41})  above is still valid because its right-hand side is then greater than $1$.
\end{proof}

\section{Concluding remarks}
\begin{enumerate}[label=(\roman*)]
\item Theorem \ref{thm:newthm41} is slightly weaker than Theorem 4.1 in \cite{HT}, because of $(\log n)^2$ instead of $\log n$ and of the extra term ${21}/{\Delta}$.

The supplementary $\log n$ comes  from the different version of the KMT Theorem (strong embedding) that is used compared to the one in \cite{HT} (weak embedding for each coordinate).  Moreover, the proof that they can be built on a same probability space with the right covariance has been in question. More precisely, in \cite{HT}, the authors use the KMT Theorem to get for each $i\in \ens{m}$ a probability space with a version of the random walk $\widetilde{B^n_i}$ and a Brownian motion approximating it, but it is not clear how this implies that on some probability space, there are versions of $\widetilde{B^n_1},\dots,\widetilde{B^n_m}$ satisfying the additional covariance requirement $\widetilde{B^n_1}+\dots+\widetilde{B^n_m}=0$. However, using Kiefer's version of the KMT approximation, we have a probability space with $(U_i)_{1\leq i\leq n}$ uniform on $[0,1]$ and with $\widetilde{B^n_1},\dots,\widetilde{B^n_m}$ well defined on it (and satisfying the right covariance relation) as shown in the proof of Lemma \ref{KMTsup}.

Note that to the best of our knowledge, it is still an open problem to know whether or not the extra $\log n$ or $(\log n)^2$ factors in the strong embedding could be improved. Note, nevertheless, the rate $1/\sqrt{n}$ in the binary uniform case (see \cite[Theorem 5.1]{HT}).

The extra term ${21}/{\Delta}$ is needed because if one picks $\Delta=p_{\max} - p_{2nd}$ very small, there is no hope, for any fixed $m$, to have a sequence $(A^m_n)_{n\geq 2}$ converging to zero such that for any distribution $p_1,\dots,p_m$, any $n\geq 2$ and $x\in \mathbb{R}$,
\begin{equation}
\left|\mathds{P}\left(\frac{LI_n-np_{\max}}{\sqrt{n p_{\max}}}\geq x\right)-\mathds{P}\left(J_k\geq x\right)\right| \leq A^m_n.
\end{equation}
Indeed, assume it is the case for, say, $m=2$. Let us apply this bound to \begin{itemize}
\item Case a) $p_1=p_2=1/2$, and denote $\left(LI_n-np_{\max}\right)/\sqrt{n p_{\max}}$ by $Z_n$
\item Case b) $p_1=1/2+1/2^n,p_2=1/2-1/2^n$, and denote $\left(LI_n-np_{\max}\right)/\sqrt{n p_{\max}}$ by $Z'_n$.
\end{itemize}
We get 
\begin{equation*}
\mathds{P}\left(Z_n\geq x\right)-\mathds{P}\left(J_2\geq x\right)-\mathds{P}\left(Z'_n\geq x\right)+\mathds{P}\left(J_1\geq x\right) \xrightarrow[n\to \infty]{} 0.
\end{equation*}
There is a coupling such that with high probability, the letters $X_1,\dots, Y_n$ in case a) are all equal to the letters in case b), hence 
\begin{equation*}
\mathds{P}\left(Z_n\geq x\right)-\mathds{P}\left(Z'_n\geq x\right)\xrightarrow[n\to \infty]{} 0.
\end{equation*}
Furthermore, $J_1$ has distribution $\mathcal{N}\left(0,\frac{1}{4}-\frac{1}{2^{n+1}}\right)$ so 
\begin{equation*}
\mathds{P}\left(J_1\geq x\right) \xrightarrow[n\to \infty]{} \mathds{P}\left(Z_{1/4}\geq x\right),
\end{equation*}
where $Z_{1/4}\sim \mathcal{N}\left(0,1/4\right)$. Putting together these three limits, we get
\begin{equation*}
\mathds{P}\left(J_2\geq x\right)=\mathds{P}\left(Z_{1/4}\geq x\right),
\end{equation*}
which means that the limiting law for the uniform binary case is normal.  However, this is known to be false. So, $(A^m_n)_{n\ge 2}$ has to depend on the distribution.  We can actually find a contradiction as soon as $\Delta =o\left(1/\sqrt{n}\right)$. Our bound \eqref{newthm41}, on the other hand, is not exposed to this kind of cases because if the right-hand side is less than $1$ then $\Delta\geq {\log n}/{\sqrt{n}}$. So Theorem 4.1 in \cite{HT} only holds for $n$ large enough, and not for all $n\geq 2$.

\item As before, let us no longer consider $m,k$ and the distribution $p_1,\dots,p_m$ to be 
fixed. We assume that both $m$ and $k$ converge to infinity with $n$, and aim to find for which sequences we have for all $\varepsilon>0$,   
\begin{equation}\label{limitgoal2}
\mathds{P}\left(\left| \left(\frac{LI_n-np_{\max}}{\sqrt{np_{\max}}}-2\sqrt{k}\right)k^{1/6}-\left(J_{k}-2\sqrt{k}\right)k^{1/6}\right|\geq \varepsilon\right) \xrightarrow[n\to \infty]{}  0.
\end{equation}

Then, since $\left(J_{k}-2\sqrt{k}\right)k^{1/6}$ has the same distribution as $\left(J_{1,k}-2\sqrt{k}\right)k^{1/6}+k^{1/6}\sqrt{\left(1-kp_{\max}\right)/k}Z$, it will converge in distribution to $F_{TW}$, implying that $\left(\left(LI_n-np_{\max}\right)/\sqrt{n p_{\max}}-2\sqrt{k}\right)k^{1/6}$ converges to $F_{TW}$ as well. 

Applying the Lemma \ref{lemnonuni} and the bound \eqref{bound2} lead to

\begin{equation}
\mathds{P}\left(\left| \left(\frac{LI_n-np_{\max}}{\sqrt{n p_{\max}}}-2\sqrt{k}\right)k^{1/6}-\left(J_{k}-2\sqrt{k}\right)k^{1/6}\right|\geq \frac{(\log n)^2}{\sqrt{n p_{\max}}}\left(\frac{21}{\Delta}+C_2 m\right)\right)\leq \frac{4m}{n^{\alpha}}.
\end{equation}

Taking $m=o((n/(\log n)^4)^{3/10})$ (as previously done in the uniform case) and $m=o(\sqrt{n p_{\max}}\Delta/(\log n)^2)$, then \eqref{limitgoal2} follows. Note that in particular, when the conditions of Theorem 6 in \cite{BH2} are satisfied, the condition \linebreak$m=o(\sqrt{n p_{\max}}\Delta/\log n)$ follows. This is not enough to conclude, first, we need $(\log n)^2$ instead of $\log n$, and more importantly, the first condition $m=o((n/(\log n)^4)^{3/10})$ is missing. This is an omission in \cite{BH2}, because as it is, there is no condition on $m$ and this leads to a counterexample.  Indeed,  let $k=n^{1/9}, p_{\max}=n^{-2/3}, m=2^n+k, p_{2nd}=\left(1-n^{5/9}\right)/2^n$, it is easy to check that the conditions of Theorem 6 there, hold true but its conclusion does not. The two conditions on $m$ we give above do fix this issue.

\item To finish these notes, let us investigate the convergence of moments which, in particular, will provide a speed of convergence result in the distance $W_p$, $p\geq 1$, given by
$$W_p(X,Y):=\inf_{\substack{X' \text{ has same law as } X\\Y'  \text{ has same law as }Y}}\left(\mathds{E} \left|X'-Y'\right|^p\right)^{1/p}.$$
Let us start with the uniform case.  We have seen that
\begin{equation}
\mathds{P}\left(\left|T_{k,m}-J_{k,m}\right|\geq 2k(m-k+1)\sqrt{m}C\frac{\alpha(\log n)^2}{\sqrt{n}}\right)\leq \frac{2}{n^\alpha}.  
\end{equation}
In particular, taking $\alpha=p$, and setting 
$\varepsilon_{n,m,k}:=2k(m-k+1)\sqrt{m}C{p(\log n)^2}/{\sqrt{n}}$, we have

\begin{align*}
\mathds{E} \left|T_{k,m}-J_{k,m}\right|^p & \leq \mathds{E}\left( \left|T_{k,m}-J_{k,m}\right|^p\mathds{1}_{\left|T_{k,m}-J_{k,m}\right|\leq \varepsilon(n,m,k)}+\left|T_{k,m}-J_{k,m}\right|^p\mathds{1}_{\varepsilon(n,m,k)<\left|T_{k,m}-J_{k,m}\right|<2m\sqrt{n}}\right.\\
&\hspace{35ex}\left.+\left|T_{k,m}-J_{k,m}\right|^p\mathds{1}_{2m\sqrt{n}\leq\left|T_{k,m}-J_{k,m}\right|}\right)\\
&\leq \varepsilon(n,m,k)^p+(2m\sqrt{n})^p\frac{2}{n^p}+\mathds{E}\left|2J_{k,m}\right|^p\mathds{1}_{m\sqrt{n}\leq\left|J_{k,m}\right|},
\end{align*}
(since $\left|T_{k,m}\right|\leq \sqrt{n}\leq m\sqrt{n}$).  Using, once more,  \eqref{thm3} and an integration by parts, one gets for some constant $C(p)$ 
\begin{equation}
\mathds{E}\left|2J_{k,m}\right|^p\mathds{1}_{m\sqrt{n}\leq\left|J_{k,m}\right|}\leq C(p)\left(\frac{m}{\sqrt{n}}\right)^p, 
\end{equation}
and therefore, for some constant $C'(p)$, 
\begin{equation}
W_p(T_{k,m},J_{k,m})\leq \mathds{E} \left|T_{k,m}-J_{k,m}\right|^p \leq C'(p)\varepsilon(m,n,k).
\end{equation}

\item Theorem \ref{thm:newthm41} can be generalized to longest common and increasing subsequences. We use the notations and framework in \cite{DH1}. Additionally, let $Z_n={\left(LCI_n-n e_{\max}\right)}/{\sqrt{n}}$ and let $Z$ be its limiting  distribution, that is, either $Z^a$ or $Z^b$ depending on the distributions $p^X, p^Y$. The proof of Lemma 2.2 in \cite{DH1}, taking $\eta=1/12$, holds for the sequence $(\varepsilon_n)=C_1(m,p^X,p^Y)n^{-1/8}$, where $C_1(m,p^X,p^Y)$ is a constant depending on $m,p^X,p^Y$. To study the cases a) and b) under the same umbrella, let us define the function $L:\left(\mathbb{R}^m\right)^2\rightarrow \mathbb{R}$ by: for all $(v^X,v^Y)\in \left(\mathbb{R}^m\right)^2$, in case a), $L(v^X,v^Y)=\sum_{i=1}^m v^X$, and in case b), $L(v^X,v^Y)=\mathfrak{m}(v^X,v^Y)$. It is not hard to see, from the expression of $\mathfrak{m}$ in Lemma 1.5 in \cite{DH1}, that in each case, there is a constant $C_2(m,p^X,p^Y)$ such that $L$ is $C_2(m,p^X,p^Y)$-Lipschitz for the $\ell^1$-distance on $\left(\mathbb{R}^m\right)^2$. Now, as a consequence of Lemma 2.2 there, it follows that if $B^{n,X}, B^{n,Y}$ are in $E_n^{1/12}$, then 
\begin{equation}
\left| Z_n - \max_{\lambda\in U}L(V^{n,X},V^{n,Y})\right|\leq \frac{C_1(m,p^X,p^Y)}{n^{1/8}},
\end{equation}
where $U=J$ in case a) and $U=K_{\Lambda^2}$ in case b). Therefore,
\begin{equation}\label{firstbound}
\mathds{P}\left(\left| Z_n - \max_{\lambda\in U}L(V^{n,X},V^{n,Y})\right|>\frac{C_1(m,p^X,p^Y)}{n^{1/8}}\right)\leq 1-\mathds{P}\left(A^{1/12}_n\right)\leq \frac{C(m)}{n},
\end{equation}
where $C(m)$ is a constant (recalling (1.9) in \cite{DH1}). Furthermore, from the expression of the limiting distribution in Theorem 2.1, we see that in any case a) or b), the limiting distribution may be written as $L(B^X, B^Y)$. We construct, with our Lemma 2.1, two Brownian motions $\widehat{B^X_n}, \widehat{B^Y_n}$ "close" to $B^{n,X}, B^{n,Y}$, on the same probability space (this is possible by applying the lemma twice, and then taking the product space). Let $\widehat{Z_n}=\max_{\lambda\in U}L(\widehat{V^{n,X}},\widehat{V^{n,Y})}$, it has same distribution as the limiting distribution $Z$, and 
 \begin{multline}
 \mathds{P}\left(\left|\max_{\lambda\in U}L(V^{n,X},V^{n,Y})-\widehat{Z_n}\right|> 2mC_2(m,p^X,p^Y)\varepsilon\right)\leq \mathds{P}\left(\sup_{\substack{i \in\ens{m} \\ 0\leq t \leq 1}} \left| \sigma_i \widehat{B^{n,X}_i}(t)-\sigma_i B^{n,X}_i(t)\right| \geq \varepsilon \right)\\
+\mathds{P}\left(\sup_{\substack{i \in\ens{m} \\ 0\leq t \leq 1}} \left| \sigma_i \widehat{B^{n,Y}_i}(t)-\sigma_i B^{n,Y}_i(t)\right| \geq \varepsilon \right),
 \end{multline}
so in particular, applying Lemma 2.1,
\begin{equation}\label{secondbound}
\mathds{P}\left(\left|\max_{\lambda\in U}L(V^{n,X},V^{n,Y})-\widehat{Z_n}\right|> 2mC_2(m,p^X,p^Y)C\frac{\alpha(\log n)^2}{\sqrt{n}}\right)\leq \frac{4}{n^\alpha}.
\end{equation}
Putting together \eqref{firstbound} and \eqref{secondbound}, with $\alpha=1$, gives
\begin{equation}\label{thirdbound}
\mathds{P}\left(\left| Z_n - \widehat{Z_n}\right|> \frac{C_1(m,p^X,p^Y)}{n^{1/8}}+2mC_2(m,p^X,p^Y)C\frac{(\log n)^2}{\sqrt{n}}\right)\leq \frac{C(m)+4}{n}.
\end{equation}

If, just as in the single subsequence case, a bound $D(m,p^X,p^Y)$ on the density of $Z$ is possible, then we can conclude in the same way: For all $x\in \mathbb{R}$,
\begin{align*}
\left|\mathds{P}\left(Z_n\geq x\right)-\mathds{P}\left(Z\geq x\right)\right|&\leq \mathds{P}\left(\left| Z_n - \widehat{Z_n}\right|> \frac{C_1(m,p^X,p^Y)}{n^{1/8}}+2mC_2(m,p^X,p^Y)C\frac{(\log n)^2}{\sqrt{n}}\right)\\
& +\mathds{P}\left(\left|\widehat{Z_n}-x\right|\leq \frac{C_1(m,p^X,p^Y)}{n^{1/8}}+2mC_2(m,p^X,p^Y)C\frac{(\log n)^2}{\sqrt{n}}\right)\\
&\leq \frac{C(m)+4}{n}+D(m,p^X,p^Y)\left(\frac{C_1(m,p^X,p^Y)}{n^{1/8}}+2mC_2(m,p^X,p^Y)C\frac{(\log n)^2}{\sqrt{n}}\right)\\
&\leq \frac{C_3(m,p^X,p^Y)}{n^{1/8}}.
\end{align*}
Note that the exponent $-1/8$ can be improved taking a smaller $\eta$, but only up to $-1/4$. This is because in \cite{DH1}, the focus was to get  convergence in distribution, rather than a tight bound. It might even be possible to get $(\log n)^2/\sqrt{n}$ instead.

\end{enumerate}

\noindent
{\bf Acknowledgements:}  Many thanks to Jon Wellner for his bibliographical help on KMT approximations, to Ryan O'Donnell and Boris Bukh for their inquiries 
on previous results, and to the ICTS in Bengaluru as well as the GESDA program at IHP in Paris for their hospitality and support while part of this research was carried out.

\bibliographystyle{abbrv}
\bibliography{biblio}
\end{document}